\RequirePackage{etex}
\documentclass[pdftex]{amsart}

\usepackage{amssymb,braket}

\usepackage{tikz}
\usepackage{pxpgfmark}
\usetikzlibrary{positioning,calc,backgrounds,arrows,graphs,quotes}

\title[Product Formula of Artin Symbols]{Product Formula of Artin Symbols in Non-abelian Extensions}
\author[S.~Sasaki]{Sosuke Sasaki}
\address{Research Institute for Science and Engineering,
Waseda University,
3-4-1 Okubo, Shinjuku-ku, Tokyo 169-8555, Japan}
\email{sosuke\_s5@asagi.waseda.jp}
\subjclass[2020]{11R37 (Primary) 11S31, 11R29, 11R27, 11R11, 05C22 (Secondary)}

\newtheorem{lemma}{Lemma}
\newtheorem{proposition}{Proposition}
\newtheorem{theorem}{Theorem}
\newtheorem{corollary}{Corollary}
\newtheorem{example}{Example}
\theoremstyle{definition}
\newtheorem{definition}{Definition}

\newcommand{\Z}{\mathbb{Z}}
\newcommand{\Q}{\mathbb{Q}}
\newcommand{\F}{\mathbb{F}}
\newcommand{\p}{\mathfrak{p}}
\renewcommand{\P}{\mathfrak{P}}
\newcommand{\ab}{\mathrm{ab}}
\newcommand{\art}[2]{\genfrac{(}{)}{}{}{#1}{#2}}
\DeclareMathOperator{\Gal}{Gal}

\usepackage{autonum}

\begin{document}

\begin{abstract}
  The product formula of Artin symbols (norm residue symbols) is an important equality
  that connects local and global class field theory.
  Usually,
  the product formula of Artin symbols is considered
  in abelian extensions of global fields.
  In this paper, however, the product is considered in non-abelian extensions
  such that each symbol is well-defined.
  As an application, some properties on fundamental units of real quadratic fields
  are obtained and will be presented here.
\end{abstract}

\maketitle

\section{Introduction}

Let $K$ be a number field (of finite degree) and
$L$ be an (possibly infinite) abelian extension of $K$.
The local Artin symbols
\[
  (\,\cdot\,,L_\P/K_\p)\colon K_\p^\times\longrightarrow\Gal(L_\P/K_\p)
\]
are defined for places $\p$ of $K$ and places $\P$ of $L$ above $\p$.
$\Gal(L_\P/K_\p)$ is naturally embedded in $\Gal(L/K)$
and its image coincides with the decomposition group $Z_\P(L/K)$.
We therefore consider $(\,\cdot\,,L_\P/K_\p)$ to take values in $Z_\P(L/K)$.
For $a\in K^\times$ we define
\[
  \art{a,L/K}{\p}=(a,L_\P/K_\p)\in Z_\P(L/K)\subseteq\Gal(L/K)
\]
which does not depend on a choice of $\P$ above $\p$:
\[
  \begin{tikzpicture}[auto,scale=2]
    \node(Kp)at(0,1){$K_\p^\times$};
    \node(K)at(0,0){$K^\times$};
    \node(LKp)at(2,1){$\Gal(L_\P/K_\p)$};
    \node(LK)at(2,0){$\Gal(L/K)$.};
    \draw[->](Kp)--node{$(\ \cdot\ ,L_\P/K_\p)$}(LKp);
    \draw[->](K)--node{$\art{\ \cdot\ ,L/K}{\p}$}(LK);
    \draw[right hook->](K)--(Kp);
    \draw[left hook->](LKp)--(LK);
  \end{tikzpicture}
\]

Now we state the well-known product formula:
\begin{proposition}
  \[
    \prod_\p\art{a,L/K}{\p}=1
  \]
  for $a\in K^\times$, where $\p$ runs through all the places of $K$.
\end{proposition}
If $L/K$ is finite, the product can be regarded as finite
since $\art{a,L/K}{\p}=1$ when $\p$ does not ramify in $L/K$ and
$a$ is a $\p$-adic unit.
Hence the product makes sense even if $L/K$ is infinite,
since the restrictions of it to arbitrary intermediate fields of $L/K$
finite over $K$ are determined.

In the following, infinite products in infinite Galois groups will appear
several times,
they are well-defined in the same way as above.
Although it is possible to argue only with finite extensions,
it is simpler to argue with infinite extensions.

In this paper we consider the product of Artin symbols
in non-abelian extensions.
Although the product is not equal to identity in general,
it can be computed explicitly in some cases.
In particular, in the case of $K=\Q$ and $a=-1$,
the product can be nicely expressed by using graph theory.

In \cite{MR720859},
the theory of genus fields and central extensions is applied to
investigate the class groups of number fields.
We can generalize it and make it more prospective
by applying our theory of products of Artin symbols.

\section{Product of local Artin symbols in non-abelian extensions}

Let $K$ be a number field and fix $a\in K^\times$.
Let $L$ be a (possibly infinite) Galois extension of $K$ satisfying the following two conditions:
\begin{enumerate}
  \item The decomposition groups $Z_\P(L/K)$ are abelian for all places $\P$ of $L$.
  \label{it:cond1}
  \item The image of $(a,L_\P/K_\p)$ in $\Gal(L/K)$ is
  in the center of $\Gal(L/K)$ for each place $\P\mid\p$ of $L/K$.
  \label{it:cond2}
\end{enumerate}

Note that (\ref{it:cond2}) is a condition depending on $a$.
Obviously, any abelian extensions of $K$ satisfy the above conditions.

From the condition (\ref{it:cond1}), the local Artin symbols $(a,L_\P/K_\p)$ are defined.
It is well-known that $(a,L_{\sigma\P}/K_\p)=\sigma(a,L_\P/K_\p)\sigma^{-1}$ for $\sigma\in\Gal(L/K)$,
and from the condition (\ref{it:cond2}), it equals $(a,L_\P/K_\p)$.
It means that $(a,L_\P/K_\p)\in\Gal(L/K)$ does not depend on a choice of $\P$ above $\p$.
Now we define
\[
  \art{a,L/K}{\p}=(a,L_\P/K_\p)\in\Gal(L/K).
\]

\begin{definition}
  For $L$ satisfying the above conditions, we write
  \[
    \pi_{L/K}(a)=\prod_\p\art{a,L/K}{\p}
  \]
  where $\p$ runs through all the places of $K$.
\end{definition}
Our purpose is to examine $\pi_{L/K}(a)$.
The product is infinite in general,
but it is well-defined as with the note in the previous section.
Also the product is commutative from the condition (\ref{it:cond2}).
We know that $\pi_{L/K}(a)$ is in the commutator group $[\Gal(L/K),\Gal(L/K)]$,
because the restriction of $\pi_{L/K}(a)$ to the maximal abelian subextension of $L/K$ becomes identity
from the usual product formula.

We consider what field $L$ satisfies the conditions (\ref{it:cond1}) and (\ref{it:cond2}).
It is easy to see that if $L$ satisfies (\ref{it:cond1}) and (\ref{it:cond2})
then any intermediate field of $L/K$ Galois over $K$ also does.

\begin{proposition}
  For each $K$ and $a\in K^\times$,
  there is a unique maximal (in a fixed algebraic closure of $\Q$) field $L$
  satisfying the above two conditions.
\end{proposition}
\begin{proof}
  Let $\Omega$ be the absolute Galois group of $K$.

  Let $\Omega_1$ be the closed subgroup
  \[
    \Braket{[Z_\P(\overline{\Q}/K),Z_\P(\overline{\Q}/K)]\mid\P\text{ : place of }\overline{\Q}}
  \]
  of $\Omega$,
  topologically generated by the commutator subgroups of all the decomposition groups of $\Omega$.
  The corresponding field $K_1$ of $\Omega_1$
  is the maximal field satisfying the condition (\ref{it:cond1}).

  Let $\Omega_a$ be the closed subgroup
  \[
    \Braket{[(a,(K_1)_\P/K_\p),\Omega_1]\mid\P\mid\p\text{ : place of }K_1/K}
  \]
  of $\Omega_1$.
  The corresponding field $K_a$ of $\Omega_a$
  is the maximal field satisfying the conditions (\ref{it:cond1}) and (\ref{it:cond2}).
\end{proof}

\begin{definition}
  $K_a$ denotes the maximal field satisfying the above conditions (\ref{it:cond1}) and (\ref{it:cond2}).
  For a prime $l$,
  $K_a^{(l)}$ denotes the maximal elementary abelian $l$-subextension of $K^\ab$ in $K_a$,
  i.e., the field corresponding to $[\Lambda_a,\Lambda_a]^l[[\Lambda_a,\Lambda_a],\Lambda_a]$
  in $\Lambda_a=\Gal(K_a/K)$.
\end{definition}

In the following, we consider products of Artin symbols
primarily in $K_a^{(2)}/K$.
If we can calculate $\pi_{K_a^{(2)}/K}(a)$,
we also obtain $\pi_{L/K}(a)=\pi_{K_a^{(2)}/K}(a)|_L$ for any intermediate field $L$ of $K_a^{(2)}/K$
with $L/K$ Galois.

First we have to consider how we represent $\pi_{K_a^{(2)}/K}(a)$.
Let $\Lambda_a^{(2)}=\Gal(K_a^{(2)}/K)$
then $\pi_{K_a^{(2)}/K}(a)\in[\Lambda_a^{(2)},\Lambda_a^{(2)}]$.
Note that the map
\begin{align}
  (\Lambda_a^{(2)})^\ab\times(\Lambda_a^{(2)})^\ab&\longrightarrow[\Lambda_a^{(2)},\Lambda_a^{(2)}]\\
  %(\sigma[\Lambda_a^{(2)},\Lambda_a^{(2)}],\tau[\Lambda_a^{(2)},\Lambda_a^{(2)}])&\longmapsto[\sigma,\tau]=\sigma\tau\sigma^{-1}\tau^{-1}
  (\sigma^\ab,\tau^\ab)&\longmapsto[\sigma,\tau]=\sigma\tau\sigma^{-1}\tau^{-1}
\end{align}
is well-defined and bilinear.

Fix a family
\[
  \set{\sigma_v|v\in V}\subseteq\Lambda_a^{(2)}
\]
with some index set $V$, such that
$\set{[\sigma_v,\sigma_w]|v,w\in V}$ generates topologically $[\Lambda_a^{(2)},\Lambda_a^{(2)}]$,
and $\set{v\in V|\sigma_v\notin\Theta}$ is finite for any open subgroup $\Theta$ of $\Lambda_a^{(2)}$.
Using them we can write
\[
  \pi_{K_a^{(2)}/K}(a)=\prod_{\{v,w\}\in[V]^2}[\sigma_v,\sigma_w]^{c_{v,w}}
  \label{eq:cvw}
\]
where $c_{v,w}\in\F_2$ and $[V]^2$ means the set of all subsets of cardinality two of $V$.
The product is infinite in general, but the product modulo $[\Theta,\Lambda_a^{(2)}]$
can be regarded as finite for any open subgroup $\Theta$ of $\Lambda_a^{(2)}$.
We want to calculate $c_{v,w}$, but in general
the choice of $(c_{v,w})_{\{v,w\}}$ is not unique
because $([\sigma_v,\sigma_w])_{\{v,w\}}$ is not a basis of $[\Lambda_a^{(2)},\Lambda_a^{(2)}]$.
Hence we have to know the ambiguity of the choice of $(c_{v,w})_{\{v,w\}}$.

Let $\F_2^{[V]^2}$ be the direct product, i.e.,
the vector space over $\F_2$ consists of all $\F_2$-valued functions on $[V]^2$.
A subset of $[V]^2$ can be regarded as an element of $\F_2^{[V]^2}$,
identifying it with its characteristic function.
We can view $\F_2^{[V]^2}$ as a topological group with
the product topology of the discrete space $\F_2$.
$(c_{v,w})_{\{v,w\}}$ is an element of $\F_2^{[V]^2}$.

Let $\F_2^{\oplus[V]^2}$ be the group consists of
all finite subsets of $[V]^2$ with the symmetric difference as an operation.
We view it with the discrete topology.

It is well-known that $\F_2^{[V]^2}$ and $\F_2^{\oplus[V]^2}$
are mutually Pontryagin dual via the pairing
\[
  \F_2^{[V]^2}\times\F_2^{\oplus[V]^2}\longrightarrow\F_2;\;(F,G)\longmapsto\sum_{e\in G}F(e).
\]

\begin{definition}
  For an intermediate field $L$ of $K_a^{(2)}/K$,
  define a surjective homomorphism of topological groups
  \begin{align}
    \mathcal{F}_{L/K}(a)\colon\F_2^{[V]^2}&\longrightarrow[\Gal(L/K),\Gal(L/K)]\\
    d&\longmapsto\prod_{\{v,w\}\in[V]^2}[\sigma_v|_L,\sigma_w|_L]^{d_{v,w}}
  \end{align}
  where $d_{v,w}$ means $d(\{v,w\})$ and $[\Lambda_a^{(2)},\Lambda_a^{(2)}]$
  is given the Krull topology.
  Let $\mathcal{A}_{L/K}(a)\subseteq\F_2^{[V]^2}$ be the kernel of $\mathcal{F}_{L/K}(a)$,
  and $\mathcal{I}_{L/K}(a)\subseteq\F_2^{\oplus[V]^2}$ be the annihilator of $\mathcal{A}_{L/K}(a)$,
  i.e.,
  \[
    \mathcal{I}_{L/K}(a)=\Set{G\in\F_2^{\oplus[V]^2}|\sum_{e\in G}F(e)=0\quad\forall F\in\mathcal{A}_{L/K}(a)}.
  \]
\end{definition}
Note that these definitions depend on the family $(\sigma_v)_{v\in V}$,
but we would not change the choice of $(\sigma_v)_{v\in V}$.

We can easily deduce that
\[
  \widehat{\mathcal{I}_{L/K}(a)}\simeq[\Gal(L/K),\Gal(L/K)]
\]
naturally via the pairing
\begin{align}
  [\Gal(L/K),\Gal(L/K)]\times\mathcal{I}_{L/K}(a)&\longrightarrow\F_2\\
  \Bigl(\prod_{\{v,w\}\in[V]^2}[\sigma_v|_L,\sigma_w|_L]^{d_{v,w}},I\Bigr)&\longmapsto\sum_{\{v,w\}\in I}d_{v,w}
\end{align}
where $\widehat{\mathcal{I}_{L/K}(a)}$ is the dual of $\mathcal{I}_{L/K}(a)$.
Also it is obvious that if $L\subseteq L'$ then
$\mathcal{A}_{L/K}(a)\supseteq\mathcal{A}_{L'/K}(a)$ and
$\mathcal{I}_{L/K}(a)\subseteq\mathcal{I}_{L'/K}(a)$.

$\mathcal{A}_{L/K}(a)$ represents
the ambiguities of the choice of $(c_{v,w})_{\{v,w\}}$ and
$\mathcal{I}_{L/K}(a)$ represents
what linear combinations of $c_{v,w}$ are invariant under
different choices of $(c_{v,w})_{\{v,w\}}$.
For each $I\in\mathcal{I}_{L/K}(a)$,
the value $\sum_{\{v,w\}\in I}c_{v,w}\in\F_2$ is uniquely determined.

\section{Determination of the invariants}

Calculating $\pi_{L/K}(a)$ on a specific number field $L$
gives some information about $(c_{v,w})_{\{v,w\}}$.
We shall consider dihedral extensions of degree $8$.

The following holds as one of the embedding problems:
\begin{proposition}[\cite{Kiming1990}]
  Let $K$ be a number field and
  $K(\sqrt\mu,\sqrt\nu)\ (\mu,\nu\in K^\times)$ be a biquadratic extension of $K$.
  Then the following are equivalent:
  \begin{enumerate}
    \item There is a dihedral extension $L/K$ of degree $8$ containing $K(\sqrt\mu,\sqrt\nu)$
    in which $\Gal(L/K(\sqrt{\mu\nu}))$ is cyclic.
    \item Hilbert symbol $\art{\mu,\nu}{\p}$ equals $1$ for each place $\p$ of $K$.
  \end{enumerate}
\end{proposition}

\begin{figure}[ht]
  \centering
  \begin{tikzpicture}[scale=1.5]
    \node(Q)at(0,-2){$K$};
    \node(m)at(-1,-1){$K(\sqrt\mu)$};
    \node(n)at(1,-1){$K(\sqrt\nu)$};
    \node(mn)at(0,-1){$K(\sqrt{\mu\nu})$};
    \node(b)at(0,0){$K(\sqrt\mu,\sqrt\nu)$};
    \node(M1)at(-2,0){$M_1$};
    \node(M2)at(-1,0){$M_2$};
    \node(N1)at(1,0){$N_1$};
    \node(N2)at(2,0){$N_2$};
    \node(L)at(0,1){$L$};
    \foreach \a/\b in{m/Q,n/Q,mn/Q,b/m,b/n,b/mn,M1/m,M2/m,N1/n,N2/n,L/b,L/M1,L/M2,L/N1,L/N2}
    \draw(\a)--(\b);
  \end{tikzpicture}
  \caption{Intermediate fields of $L/K$}
\end{figure}
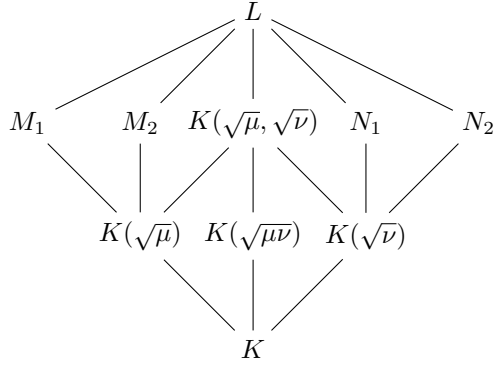

Adding the condition $L\subseteq K_a$ to this proposition, we obtain the following:
\begin{proposition}
  \label{prop:d8}
  Let $K$ be a number field and $a\in K^\times$.
  Let $K(\sqrt\mu,\sqrt\nu)\ (\mu,\nu\in K^\times)$ be a biquadratic extension of $K$.
  Then the following are equivalent:
  \begin{enumerate}
    \item There is a dihedral extension $L/K$ of degree $8$ contained in $K_a^{(2)}$ and
    containing $K(\sqrt\mu,\sqrt\nu)$
    in which $\Gal(L/K(\sqrt{\mu\nu}))$ is cyclic.
    \item $\art{\mu,\nu}{\p}=\art{a,\mu}{\p}=\art{a,\nu}{\p}=1$ and
    $\p$ splits in either of $K(\sqrt\mu),K(\sqrt\nu),K(\sqrt{\mu\nu})$,
    for each place $\p$ of $K$.
    \label{it:d82}
  \end{enumerate}
\end{proposition}
\begin{proof}
  Recall that $L\subseteq K_a$ if and only if,
  all the decomposition groups of $L/K$ are abelian and
  all the local Artin symbols at $a$ are in the center of $\Gal(L/K)$.
  A subgroup of $\Gal(L/K)$ is abelian if and only if it is a proper subgroup,
  and the center of $\Gal(L/K)$ equals $\Gal(L/K(\sqrt\mu,\sqrt\nu))$.
  From these we obtain the claim.
\end{proof}
Note that for $\p$ which does not ramify in $K(\sqrt\mu,\sqrt\nu)$,
$\p$ splits in either of $K(\sqrt\mu),K(\sqrt\nu),K(\sqrt{\mu\nu})$
because the decomposition group is cyclic.

\begin{proposition}
  \label{prop:trip}
  Let $K,a,\mu,\nu,L$ satisfy the condition of Proposition~\ref{prop:d8}.

  Let $(\sigma_v)_{v\in V}$ be a family as in the previous section.
  Define disjoint three finite subsets of $V$ as follows:
  \begin{align}
    V_1&=\set{v\in V|\sigma_v\sqrt\mu=\sqrt\mu,\,\sigma_v\sqrt\nu=-\sqrt\nu},\\
    V_2&=\set{v\in V|\sigma_v\sqrt\mu=-\sqrt\mu,\,\sigma_v\sqrt\nu=-\sqrt\nu},\\
    V_3&=\set{v\in V|\sigma_v\sqrt\mu=-\sqrt\mu,\,\sigma_v\sqrt\nu=\sqrt\nu}.
  \end{align}

  Then, $\mathcal{I}_{L/K}(a)$ is the group of order two
  generated by
  \[
    C=\set{\{v,w\}|v\in V_i,\,w\in V_j\quad\exists i,j\in\{1,2,3\},\,i\ne j}\in\F_2^{\oplus[V]^2}.
  \]
  In particular,
  \[
    \pi_{L/K}(a)=1\iff\sum_{\{v,w\}\in C}c_{v,w}=0
  \]
  where $(c_{v,w})_{\{v,w\}}$ is that in the equality~\eqref{eq:cvw}.
\end{proposition}
\begin{proof}
  The order of $\mathcal{I}_{L/K}(a)$ is two since
  \[
    \widehat{\mathcal{I}_{L/K}(a)}\simeq[\Gal(L/K),\Gal(L/K)]=\Gal(L/K(\sqrt\mu,\sqrt\nu)).
  \]
  One can easily check that
  for $v,w\in V$, $[\sigma_v|_L,\sigma_w|_L]$ equals the generator of $L/K(\sqrt\mu,\sqrt\nu)$
  if $\{v,w\}\in C$, and
  $[\sigma_v|_L,\sigma_w|_L]=1$ otherwise.
  Hence for $(d_{v,w})_{\{v,w\}}\in\F_2^{[V]^2}$,
  \[
    \prod_{\{v,w\}\in[V]^2}[\sigma_v|_L,\sigma_w|_L]^{d_{v,w}}=1\iff\sum_{\{v,w\}\in C}d_{v,w}=0.
  \]
  Therefore we have $\mathcal{I}_{L/K}(a)=\braket{C}$.
\end{proof}

For a quadratic extension $M/K$, we write $S_{M/K}$ (resp.~$I_{M/K},R_{M/K}$)
to be the set of the places of $K$ which splits (resp.~is inert, ramifies)
in $M$.

\begin{lemma}
  \label{lem:fund}
  Let $K$ be a number field and $a\in K^\times$.
  Let $L$ and $M$ be finite Galois extensions of $K$
  such that $K\subseteq M\subseteq L\subseteq K_a$,
  $L/M$ is abelian and $[M:K]=2$.
  Assume that there exists
  $\alpha\in M^\times$ such that $N_{M/K}(\alpha)=a$.
  Take $\sigma\in\Gal(L/K)\smallsetminus\Gal(L/M)$.
  Then
  \[
    \pi_{L/K}(a)=\prod_{\p\in S_{M/K}}\left[\sigma,\art{\alpha,L/M}{\P}\right]
  \]
  where $\P$ stands for a single place of $M$ above $\p$.
\end{lemma}
\begin{proof}
  If $\p$ splits in $M$ one has $M_\P=K_\p$, hence
  \[
    \art{a,L/K}{\p}=\art{a,L/M}{\P}=\art{\alpha,L/M}{\P}\art{\sigma\alpha,L/M}{\P}.
  \]
  If $\p$ does not split in $M$ one has $N_{M_\P/K_\p}(\alpha)=a$, hence
  \[
    \art{a,L/K}{\p}=\art{\alpha,L/M}{\P}.
  \]
  On the other hand, from the product formula
  \[
    \prod_{\p\in S_{M/K}}\art{\alpha,L/M}{\P}\art{\alpha,L/M}{\sigma\P}
    \prod_{\p\notin S_{M/K}}\art{\alpha,L/M}{\P}=1.
  \]
  Therefore we have
  \begin{align}
    \pi_{L/K}(a)&=\prod_{\p\in S_{M/K}}\art{a,L/K}{\p}\prod_{\p\notin S_{M/K}}\art{a,L/K}{\p} \\
    &=\prod_{\p\in S_{M/K}}\art{\sigma\alpha,L/M}{\P}\art{\alpha,L/M}{\sigma\P}^{-1} \\
    &=\prod_{\p\in S_{M/K}}\sigma\art{\alpha,L/M}{\sigma\P}\sigma^{-1}\art{\alpha,L/M}{\sigma\P}^{-1}.
  \end{align}
  Replacing $\sigma\P$ by $\P$ we obtain the result.
\end{proof}

\begin{proposition}
  \label{prop:main}
  Let $K,a,\mu,\nu,L$ satisfy the condition of Proposition~\ref{prop:d8} and
  let $\alpha\in K(\sqrt\mu)^\times$ be such that $N_{K(\sqrt\mu)/K}(\alpha)=a$
  (which exists from Hasse norm theorem).

  Assume that there exists $b\in K^\times$ such that
  \begin{enumerate}
    \item $v_\p(b)=0$ for non-archimedean $\p\in I_{K(\sqrt\mu)/K}\cup R_{K(\sqrt\mu)/K}$;
    \item either $v_\P(b\alpha)$ or $v_{\P'}(b\alpha)$ is even
    for places $\P,\P'$ of $K(\sqrt\mu)$ above each non-archimedean place $\p\in S_{K(\sqrt\mu)/K}$;
  \end{enumerate}
  where $v$ means the valuation.

  Then $\pi_{L/K}(a)=1$ if and only if
  \[
    \prod_{\p\in S_{K(\sqrt\mu)/K}\cap R_{K(\sqrt\nu)/K}}\art{b\alpha,\nu}{\P}=
    \prod_{\p\in R_{K(\sqrt\mu)/K}\cap R_{K(\sqrt\nu)/K}}\art{b,\nu}{\p}
  \]
  where $\P$ stands for a single place of $K(\sqrt\mu)$ above $\p$.
\end{proposition}
\begin{proof}
  Let $M=K(\sqrt\mu)$ and $N=K(\sqrt\nu)$.
  Let $\sigma$ be a generator of $\Gal(L/K(\sqrt{\mu\nu}))$.
  Since the centralizer of $\sigma$ is $\langle\sigma\rangle$,
  \[
    \left[\sigma,\art{\alpha,L/M}{\P}\right]=1\iff\art{\alpha,MN/M}{\P}=1.
  \]
  Therefore from Lemma~\ref{lem:fund},
  \[
    \pi_{L/K}(a)=1\iff\prod_{\p\in S_{M/K}}\art{\alpha,MN/M}{\P}=1,
  \]
  here for each $\p\in S_{M/K}$ we choose $\P$ above $\p$ such that $v_\P(b\alpha)$ is even.

  From the product formula in $N/K$ we have
  \[
    \prod_{\p\in S_{M/K}}\art{\alpha,MN/M}{\P}=
    \prod_{\p\in S_{M/K}}\art{b\alpha,MN/M}{\P}
    \prod_{\p\in I_{M/K}\cup R_{M/K}}\art{b,N/K}{\p}.
  \]
  By our assumption, $v_\P(b\alpha)$ is even for each non-archimedean place $\p\in S_{M/K}$ and
  $v_\p(b)=0$ for $\p\in I_{M/K}\cup R_{M/K}$.
  Hence on the right hand side of the above equality,
  we may assume $\p\in R_{N/K}$.
  Also we have $I_{M/K}\cap R_{N/K}=\emptyset$,
  since all the decomposition groups of $L/K$ are abelian,
  so either the inertia degree or the ramification index of $MN/K$ must be one
  for each place.
\end{proof}

\section{The case $K=\Q$ and $a=-1$}

We shall consider the case $K=\Q$ and $a=-1$.
In this case, the ambiguities $\mathcal{A}_{\Q_{-1}^{(2)}/\Q}(-1)$
and the invariants $\mathcal{I}_{\Q_{-1}^{(2)}/\Q}(-1)$
can be completely determined using graph theory.

First we prepare the following notation:
\begin{itemize}
  \item $\art{m}{p}$ or $(m/p)$ denotes the Legendre symbol
  for an integer $m$ and an odd prime $p\nmid m$, and
  \[
    \art{m}{2}=\begin{cases}
      +1&(m\equiv\pm1\pmod8)\\
      -1&(m\equiv\pm5\pmod8)
    \end{cases}
  \]
  for an odd $m$.
  \item $\art{m}{p}_4$ or $(m/p)_4$ denotes the quartic residue symbol
  for an integer $m$ and a prime $p\equiv1\ (4)$ with $(m/p)=1$, and
  \[
    \art{m}{2}_4=\begin{cases}
      +1&(m\equiv\pm1\pmod{16})\\
      -1&(m\equiv\pm9\pmod{16})
    \end{cases}
  \]
  for $m\equiv\pm1\ (8)$.
  \item $\varepsilon_m$ denotes a positive fundamental unit of the real quadratic field $\Q(\sqrt m)$.
  We write $N\varepsilon_m$ for the norm of $\varepsilon_m$ from $\Q(\sqrt m)$ to $\Q$.
  If $p\not\equiv3\ (4)$ is a prime which splits in $\Q(\sqrt m)$,
  we write $(\varepsilon_m/p)=\art{\varepsilon_m,p}{\p}$ where $\p$ is a place of $\Q(\sqrt m)$ above $p$.
\end{itemize}

First we define a family $\set{\sigma_p|p\in V}\subseteq\Lambda_{-1}^{(2)}=\Gal(\Q_{-1}^{(2)}/\Q)$.
From the class field theory we know that
$(\Lambda_{-1}^{(2)})^\ab$ is generated modulo square by
\[
  \art{g_p,\Q^\ab/\Q}{p}\ (p\text{ : prime}),\,
  \art{-1,\Q^\ab/\Q}{2},\,
  \art{-1,\Q^\ab/\Q}{\infty}
\]
where $g_p$ is a primitive root modulo $p$ if $p$ is odd
and $g_2=5$.
Take $\sigma_p,\sigma_{-1},\sigma_\infty$ to be lifts of these respectively.
The definition of $\Q_{-1}$ yields that
$\sigma_{-1}$ and $\sigma_\infty$ are in the center
of $\Lambda_{-1}^{(2)}$.
Similarly, if $p\equiv3\ (4)$, $\sigma_p$ is in the center
of $\Lambda_{-1}^{(2)}$
since $\art{-g_p,\Q^\ab/\Q}{p}$ is square.
Hence if we set
\[
  V=\set{p\text{ : prime}|p\not\equiv3\pmod4},
\]
we have that $[\Lambda_{-1}^{(2)},\Lambda_{-1}^{(2)}]$ is generated topologically by
$\set{[\sigma_p,\sigma_q]|p,q\in V}$.
Also $\set{p\in V|\sigma_p\notin\Theta}$ is finite for any open subgroup $\Theta\subseteq\Lambda_{-1}^{(2)}$
since $\sigma_p$ is in the inertia group above $p$.

Therefore we can write
\[
  \pi_{\Q_{-1}^{(2)}/\Q}(-1)=\prod_{\{p,q\}\in[V]^2}[\sigma_p,\sigma_q]^{c_{p,q}}
  \label{eq:qcpq}
\]
with some $c_{p,q}=c_{q,p}\in\F_2$.
Our goal in this section is to determine $\mathcal{A}_{\Q_{-1}^{(2)}/\Q}(-1)$, $\mathcal{I}_{\Q_{-1}^{(2)}/\Q}(-1)$
and values $\sum_{\{p,q\}\in I}c_{p,q}\in\F_2$ for $I\in\mathcal{I}_{\Q_{-1}^{(2)}/\Q}(-1)$.

\begin{definition}
  For a (simple, undirected, possibly infinite) graph $\Gamma=(V,E)\ (E\subseteq[V]^2)$,
  \begin{itemize}
    \item let $\mathcal{E}(\Gamma)=\F_2^E$ be the direct product,
    given the product topology of the discrete space $\F_2$;
    \item for $v\in V$ let $E(v)=\set{e\in E|v\in e}$ be the set of edges
    which are incident on $v$, and
    let $\mathcal{B}(\Gamma)$ be the closed subgroup of $\mathcal{E}(\Gamma)$
    topologically generated by $E(v)$ for $v\in V$;
    \item let $\mathcal{G}(\Gamma)=\F_2^{\oplus E}$ be the direct sum,
    given the discrete topology;
    \item let $\mathcal{C}(\Gamma)$ be the subgroup of $\mathcal{G}(\Gamma)$
    generated by cycles
    \[
      \{\{v_1,v_2\},\{v_2,v_3\},\dots,\{v_k,v_1\}\}
    \]
    where $v_1,\dots,v_k\in V$ are distinct and
    $\{v_1,v_2\},\dots,\{v_k,v_1\}\in E$.
  \end{itemize}
\end{definition}

Split $[V]^2$ into quadratic residue edges and non-residue edges:
\begin{align}
  E_\mathrm{R}&=\Set{\{p,q\}\in[V]^2|\art{p}{q}=+1},\\
  E_\mathrm{N}&=\Set{\{p,q\}\in[V]^2|\art{p}{q}=-1}.
\end{align}
Let $\Gamma_\mathrm{R}=(V,E_\mathrm{R})$ and $\Gamma_\mathrm{N}=(V,E_\mathrm{N})$.

\begin{lemma}
  $\mathcal{C}(\Gamma_\mathrm{N})$ is generated by triangles, i.e.,
  \[
    \mathcal{C}(\Gamma_\mathrm{N})=\Braket{\{\{p,q\},\{q,r\},\{r,p\}\}|p,q,r\in V,\,\art{p}{q}=\art{q}{r}=\art{r}{p}=-1}.
  \]
\end{lemma}
\begin{proof}
  Take distinct $p_1,\dots,p_k\in V$ with
  $(p_1/p_2)=(p_2/p_3)=\dots=(p_k/p_1)=-1$.
  Let
  \[
    C=\{\{p_1,p_2\},\{p_2,p_3\},\dots,\{p_{k-1},p_k\},\{p_k,p_1\}\}\in\mathcal{C}(\Gamma_\mathrm{N})
  \]
  be a cycle of length $k>2$.
  From Dirichlet's theorem on arithmetic progressions,
  there exists $l\in V$ such that $(p_i/l)=-1$ for $i=1,\dots,k$.
  Then $C$ equals the symmetric difference of $k$ triangles
  $C_{1,2},C_{2,3},\dots,C_{k,1}$,
  where $C_{i,j}=\{\{p_i,p_j\},\{p_j,l\},\{l,p_i\}\}$.
  Hence all cycles of $\mathcal{C}(\Gamma_\mathrm{N})$ are generated by triangles,
  which proves the lemma.
\end{proof}

\begin{theorem}
  \label{thm:graph}
  $\mathcal{A}_{\Q_{-1}^{(2)}/\Q}(-1)=\mathcal{B}(\Gamma_\mathrm{N})$ and
  $\mathcal{I}_{\Q_{-1}^{(2)}/\Q}(-1)=\mathcal{G}(\Gamma_\mathrm{R})\oplus\mathcal{C}(\Gamma_\mathrm{N})$.
\end{theorem}
\begin{proof}
  Let $p\in V$.
  From the product formula,
  \[
    \prod_q\art{p,\Q_{-1}^{(2)}/\Q}{q}\in[\Lambda_{-1}^{(2)},\Lambda_{-1}^{(2)}]
  \]
  where the product runs over all places of $\Q$,
  and $\art{p,\Q_{-1}^{(2)}/\Q}{q}$ denotes a lift of $\art{p,\Q^\ab/\Q}{q}$.
  Here if $(p/q)=1$, $\art{p,\Q_{-1}^{(2)}/\Q}{q}$ are in the center of $\Lambda_{-1}^{(2)}$
  since $\art{p,\Q^\ab/\Q}{q}$ is square.
  Similarly if $q\equiv3\ (4)$ or $q=\infty$, $\art{p,\Q_{-1}^{(2)}/\Q}{q}$ is also in the center of $\Lambda_{-1}^{(2)}$.
  Hence we have
  \[
    1=\left[\sigma_p,\prod_q\art{p,\Q_{-1}^{(2)}/\Q}{q}\right]
    =\left[\sigma_p,\art{p,\Q_{-1}^{(2)}/\Q}{p}\right]\prod_{q\in V,\,(p/q)=-1}[\sigma_p,\sigma_q],
  \]
  and $\Bigl[\sigma_p,\art{p,\Q_{-1}^{(2)}/\Q}{p}\Bigr]=1$
  since the decomposition groups of $\Lambda_{-1}^{(2)}$ are abelian.
  Therefore $\mathcal{A}_{\Q_{-1}^{(2)}/\Q}(-1)$ contains $E_\mathrm{N}(p)$ for all $p\in V$,
  which yields $\mathcal{A}_{\Q_{-1}^{(2)}/\Q}(-1)\supseteq\mathcal{B}(\Gamma_\mathrm{N})$.

  Let $p,q\in V$ with $(p/q)=1$.
  One can check that the condition (\ref{it:d82}) of Proposition~\ref{prop:d8} holds for $K=\Q,\,a=-1,\,\mu=p,\,\nu=q$.
  Hence there exists a dihedral field $L$ of degree $8$ with $L\subseteq\Q_{-1}^{(2)}$.
  From Proposition~\ref{prop:trip},
  $\{p,q\}$ is in $\mathcal{I}_{L/\Q}(-1)$,
  hence in $\mathcal{I}_{\Q_{-1}^{(2)}/\Q}(-1)$.
  This yields $\mathcal{I}_{\Q_{-1}^{(2)}/\Q}(-1)\supseteq\mathcal{G}(\Gamma_\mathrm{R})$.

  Similarly let $p,q,r\in V$ with $(p/q)=(q/r)=(r/p)=-1$.
  From Proposition~\ref{prop:d8} and \ref{prop:trip} with $K=\Q,\,a=-1,\,\mu=pq,\,\nu=pr$
  we have $\{p,q\},\{q,r\},\{r,p\}$ are in $\mathcal{I}_{L/\Q}(-1)$,
  hence in $\mathcal{I}_{\Q_{-1}^{(2)}/\Q}(-1)$.
  This yields $\mathcal{I}_{\Q_{-1}^{(2)}/\Q}(-1)\supseteq\mathcal{C}(\Gamma_\mathrm{N})$
  from the above lemma.

  Now we deduce the theorem from the next proposition.
\end{proof}

\begin{proposition}
  $\mathcal{B}(\Gamma)$ and $\mathcal{C}(\Gamma)$ are mutually
  annihilators in the duality of $\mathcal{E}(\Gamma)$ and $\mathcal{G}(\Gamma)$
  for any graph $\Gamma$.
\end{proposition}
\begin{proof}
  For a finite graph,
  see \cite[Theorem~1.9.4]{MR3644391}.
  It can be similarly shown for an infinite graph.
\end{proof}

$\mathcal{I}_{\Q_{-1}^{(2)}/\Q}(-1)$ is generated by
the singletons $\{\{p,q\}\}$ of the residue edges $\{p,q\}$,
and the triangles $\{\{p,q\},\{q,r\},\{r,p\}\}$
of the non-residue edges $\{p,q\},\{q,r\},\{r,p\}$.
Hence it suffices to calculate the sums of $c_{p,q}$
on these generators.

\begin{proposition}
  \label{prop:scholz}
  Let $p,q,r\in V$.
  \begin{enumerate}
    \item\label{it:scholz1}
    If $(p/q)=1$,
    \[
      \art{\varepsilon_p}{q}=\art{p}{q}_4\art{q}{p}_4.
    \]
    \item\label{it:scholz2}
    If $(p/q)=(q/r)=(r/p)=-1$,
    \[
      \art{\varepsilon_{pq}}{r}=-\art{pq}{r}_4\art{qr}{p}_4\art{rp}{q}_4.
    \]
  \end{enumerate}
\end{proposition}
(\ref{it:scholz1}) is known as Scholz's reciprocity law~\cite{MR1761696}.
We shall prove (\ref{it:scholz2}).
\begin{lemma}
  \label{lem:e}
  Let $m>2$ be a square-free odd integer
  such that $N\varepsilon_m=-1$.
  Then
  $\Q(\sqrt{(m/2)\varepsilon_m\sqrt m})/\Q(\sqrt m)$
  is unramified outside $m\infty$.
\end{lemma}
\begin{proof}
  It suffices to consider the ramification above $2$.
  Since $\varepsilon_m^3\in\Z[\sqrt m]$,
  we write $\varepsilon_m^3=x+y\sqrt m$ with $x,y\in\Z$.
  Since $x^2-my^2=-1$,
  we can show that any prime divisor $p$ of $my$ satisfies $p\equiv1\ (4)$.
  Thus $x$ is even, and $x\equiv0\ (4)$ if and only if $m\equiv1\ (8)$.
  Also we have $y\equiv1\ (4)$ since $y>0$.
  Therefore
  \[
    \art{m}{2}\varepsilon_m^3\sqrt m=\art{m}{2}(my+x\sqrt m)\equiv1\pmod4
  \]
  in the ring of integers of $\Q(\sqrt m)$.
  It is well-known that such a Kummer extension is unramified above $2$,
  see \cite[Theorem~4.12]{MR1761696}.
\end{proof}
\begin{lemma}
  \label{lem:quar}
  Let $p\in V$
  and $F$ be a cyclic quartic field
  which $p$ ramifies totally.
  In the case $p=2$ we also assume
  $\art{-1,F/\Q}{2}=1$.

  Then for $m\in\Z$ prime to $p$,
  \[
    \art{m,F/\Q}{p}=1\iff\art{m}{p}_4=1.
  \]
\end{lemma}
\begin{proof}
  From the local class field theory,
  the local Artin map induces a surjective homomorphism
  \[
    \Z_p^\times\longrightarrow\Gal(F/\Q).
  \]
  If $p\equiv1\ (4)$, the kernel is $(\Z_p^\times)^4$
  since $[\Z_p^\times:(\Z_p^\times)^4]=4$.
  If $p=2$, $\Z_2^\times$ can be written as the direct product $\braket{-1}\times\braket{5}$,
  hence from our assumption,
  the kernel is $\braket{-1}\times\braket{5}^4$.
  From these we have the result.
\end{proof}
\begin{proof}[Proof of Proposition~\ref{prop:scholz} (\ref{it:scholz2})]
  It is well-known that the assumption $(p/q)=-1$
  implies $N\varepsilon_{pq}=-1$.
  We may assume $\varepsilon_{pq}>1$.

  If $pq$ is odd, we put $e=(pq/2)$.
  If $p=2$ or $q=2$, we may assume $p=2$, and put $e=(q/2)$.
  Let $F=\Q(\sqrt{e\varepsilon_{pq}\sqrt{pq}})$,
  then from Lemma~\ref{lem:e}, $F/\Q$ is unramified outside $pq\infty$.

  The norm of $e\varepsilon_{pq}\sqrt{pq}$ is $pq$,
  hence $F$ is a cyclic quartic field.
  From the product formula of Artin symbols we have
  \[
    \art{r,F/\Q}{p}
    \art{r,F/\Q}{q}
    \art{r,F/\Q}{r}=1,
  \]
  here the archimedean place can be omitted since $r>0$.
  Again from the product formula,
  \[
    \art{q,F/\Q}{p}
    \art{q,F/\Q}{q}=1.
  \]
  Combining these we have
  \[
    \art{qr,F/\Q}{p}
    \art{qr,F/\Q}{q}
    \art{r,F/\Q}{r}=1.
    \label{eq:comb}
  \]
  Let $\gamma$ be the generator of $\Gal(F/\Q(\sqrt{pq}))$.
  If $p=2$, from the product formula,
  \[
    \art{-1,F/\Q}{2}
    \art{-1,F/\Q}{q}
    \art{-1,F/\Q}{\infty}=1.
  \]
  Lemma~\ref{lem:quar} shows that
  the last two factors are both trivial or both $\gamma$, hence we have
  \[
    \art{-1,F/\Q}{2}=1.
  \]
  So we can use Lemma~\ref{lem:quar} again for general $p$, which yields
  \[
    \art{qr,F/\Q}{p}
    =\begin{cases}
      1&((qr/p)_4=+1)\\\gamma&((qr/p)_4=-1).
    \end{cases}
  \]
  Since $pq$ is the norm of $\sqrt{e\varepsilon_{pq}\sqrt{pq}}$, we have
  \[
    \art{pq,F/\Q}{q}=1.
  \]
  Hence from Lemma~\ref{lem:quar},
  \[
    \art{qr,F/\Q}{q}
    =\art{p^3r,F/\Q}{q}
    =\begin{cases}
      1&((pr/q)_4(p/q)=+1)\\\gamma&((pr/q)_4(p/q)=-1).
    \end{cases}
  \]
  Also it is easy to show that
  \[
    \art{r,F/\Q}{r}=\art{r,e\varepsilon_{pq}\sqrt{pq}}{\mathfrak{r}}=\begin{cases}
      1&((\varepsilon_{pq}/r)(pq/r)_4=+1)\\
      \gamma&((\varepsilon_{pq}/r)(pq/r)_4=-1).
    \end{cases}
  \]
  where $\mathfrak{r}$ is a place of $\Q(\sqrt{pq})$ above $r$.
  Substituting these into the equality~\eqref{eq:comb}
  we obtain the result.
\end{proof}

Now we obtain the main theorem in the case of $K=\Q$ and $a=-1$:
\begin{theorem}
  \label{thm:qmain}
  Let $(c_{p,q})_{\{p,q\}}$ be that in the equality~\eqref{eq:qcpq}.
  \begin{enumerate}
    \item For $p,q\in V$ with $(p/q)=1$,
    \[
      c_{p,q}=0\iff\art{p}{q}_4\art{q}{p}_4=1.
    \]
    \label{it:qmain1}
    \item For $p,q,r\in V$ with $(p/q)=(q/r)=(r/p)=-1$,
    \[
      c_{p,q}+c_{q,r}+c_{r,p}=0\iff\art{pq}{r}_4\art{qr}{p}_4\art{rp}{q}_4=-1.
    \]
    \label{it:qmain2}
  \end{enumerate}
  More generally,
  if $I\in\mathcal{I}_{\Q_{-1}^{(2)}/\Q}(-1)$ and
  $P$ is the union set of $I$, i.e.,
  \[
    P=\set{p\in V|\{p,q\}\in I\quad\exists q\in V},
  \]
  then
  \[
    \sum_{\{p,q\}\in I}c_{p,q}=0\iff\prod_{p\in P}\art{\prod_{\{p,q\}\in I}q}{p}_4=(-1)^k
  \]
  where $k$ is the number of $\{p,q\}\in I$ with $(p/q)=-1$.
\end{theorem}
\begin{proof}
  \indent
  \begin{enumerate}
    \item From Proposition~\ref{prop:d8} with $K=\Q,\,a=-1,\,\mu=p,\,\nu=q$,
    there exists a dihedral field $L$ of degree $8$ with $L\subseteq\Q_{-1}^{(2)}$.
    Using Proposition~\ref{prop:trip} and \ref{prop:main}
    with $\alpha=\varepsilon_p,\,b=1$,
    we have
    \[
      c_{p,q}=0\iff\pi_{L/K}(-1)=1\iff\art{\varepsilon_p}{q}=1.
    \]
    Combining Proposition~\ref{prop:scholz} we get the result.
    \item Similarly we set $K=\Q,\,a=1,\,\mu=pq,\,\nu=pr,\,\alpha=\varepsilon_{pq},\,b=1$,
    from the propositions we have
    \[
      c_{p,q}+c_{q,r}+c_{r,p}=0\iff\pi_{L/K}(-1)=1\iff\art{\varepsilon_{pq}}{r}=1.
    \]
    Combining Proposition~\ref{prop:scholz} we get the result.
  \end{enumerate}

  The last part of the theorem can be shown as follows.
  Take $l\in V$ such that $(p/l)=-1$ for all $p\in P$,
  which exists from Dirichlet's theorem.
  From the above (\ref{it:qmain1}) and (\ref{it:qmain2}),
  \[
    \label{eq:c}
    \sum_{\begin{smallmatrix}
      \{p,q\}\in I\\(p/q)=1
    \end{smallmatrix}}c_{p,q}+
    \sum_{\begin{smallmatrix}
      \{p,q\}\in I\\(p/q)=-1
    \end{smallmatrix}}(c_{p,q}+c_{p,l}+c_{q,l})=0
  \]
  if and only if
  \[
    \label{eq:a4}
    \prod_{\begin{smallmatrix}
      \{p,q\}\in I\\(p/q)=1
    \end{smallmatrix}}\art{p}{q}_4\art{q}{p}_4
    \prod_{\begin{smallmatrix}
      \{p,q\}\in I\\(p/q)=-1
    \end{smallmatrix}}\art{pl}{q}_4\art{ql}{p}_4\art{pq}{l}_4=(-1)^k.
  \]
  The left-hand side of \eqref{eq:c} equals
  \[
    \sum_{\{p,q\}\in I}c_{p,q}+\sum_{p\in P}d_pc_{p,l}
  \]
  and the left-hand side of \eqref{eq:a4} equals
  \[
    \prod_{p\in P}\art{l^{d_p}\prod_{\{p,q\}\in I}q}{p}_4\art{p^{d_p}}{l}_4,
  \]
  where $d_p$ is the number of edges $\{p,q\}\in I$ with $(p/q)=-1$,
  which is even because $\mathcal{C}(\Gamma_\mathrm{N})$ is generated by cycles.
  Therefore $(l^{d_p}/p)_4(p^{d_p}/l)_4=(l/p)^{d_p/2}(p/l)^{d_p/2}=1$
  so we have done.
\end{proof}

\section{Application on central extensions}

Let $L/K$ be an abelian extension of number fields.
An extension $F$ of $L$ is called a central extension of $L/K$
if $F/K$ is Galois and $\Gal(F/L)$ is contained in the center of $\Gal(F/K)$.

Let $Z$ be the central class field of $L/K$,
i.e., the maximal central extension of $L/K$ which is unramified over $L$.
Also let $G$ be the genus field of $L/K$,
i.e., the maximal abelian extension of $K$ which is unramified over $L$.
The degree of $G/L$ is called the genus number,
which can be calculated by the famous formula \cite{nmj/1118802021}.

In \cite{MR720859}, especially in the case of $K=\Q$,
the class group of $L$ is examined by analyzing $\Gal(Z/G)$.
In this section we shall generalize it for arbitrary $K$,
by using our results on the products of Artin symbols.

Also we shall use the following notation:
\begin{itemize}
  \item $J_K$ denotes the id\`ele group of $K$.
  \item $U_K$ denotes the unit id\`ele group of $K$.
  \item $E_K$ denotes the global unit group of $K$.
\end{itemize}

\begin{theorem}
  \label{thm:cent}
  Let $L/K$ be an abelian extension of number fields and
  $Z$ be the central class field of $L/K$.
  Let $F$ be the maximal central extension of $L/K$ such that
  all the decomposition groups of $F/K$ are abelian.
  Then
  \[
    F\subseteq\bigcap_{a\in K^\times\cap NJ_L}K_a
  \]
  and the map $\pi_{F/K}$ induces isomorphisms
  \begin{align}
    K^\times\cap NJ_L\Big/NL^\times&\simeq\Gal(F/K^\ab),\\
    (E_K\cap NU_L)NL^\times\Big/NL^\times&\simeq\Gal(F/ZK^\ab)
  \end{align}
  where $N=N_{L/K}\colon J_L\to J_K$ is the norm map.
\end{theorem}
Note that $F/L$ is infinite but
$F/K^\ab$ is finite, since the commutator map induces
the well-defined surjection
\[
  \Gal(L/K)\wedge\Gal(L/K)\longrightarrow\Gal(F/K^\ab)
\]
from the exterior square of $\Gal(L/K)$, which is finite.

To prove it we prepare the following lemma:
\begin{lemma}
  Let $M$ be a finite abelian extension of $L$ with $M/K$ Galois.
  Then
  \begin{enumerate}
    \item $M/L$ is unramified iff $U_L\subseteq L^\times N_{M/L}J_M$;
    \label{it:ideleur}
    \item $M/K$ is abelian iff $N_{L/K}^{-1}K^\times\subseteq L^\times N_{M/L}J_M$;
    \label{it:ideleab}
    \item $M$ is a central extension of $L/K$ iff $A_{L/K}\subseteq L^\times N_{M/L}J_M$,
    where $A_{L/K}$ is the subgroup of $J_L$ generated by $\sigma\mathbf{a}/\mathbf{a}$
    for $\mathbf{a}\in J_L,\,\sigma\in\Gal(L/K)$;
    \label{it:idelecent}
    \item $M$ is a central extension of $L/K$ and all the decomposition group of $M/K$ are abelian
    iff $H_{L/K}\subseteq L^\times N_{M/L}J_M$,
    where $H_{L/K}$ is the kernel of $N_{L/K}\colon J_L\to J_K$;
    \label{it:idelef}
    \item For a place $\P$ of $L$, the decomposition group of $M/K$ above $\P$ is abelian
    iff $H_{L/K,\P}\subseteq L^\times N_{M/L}J_M$,
    where $H_{L/K,\P}=H_{L/K}\cap L_\P^\times$.
    \label{it:ideleloc}
  \end{enumerate}
\end{lemma}
\begin{proof}
  (\ref{it:ideleur}) is well-known.
  (\ref{it:ideleab}) is from \cite[Lemma~2.2]{MR720859}.

  \begin{enumerate}
    \item[(\ref{it:idelecent})] In the isomorphism $J_L\big/L^\times N_{M/L}J_M\simeq\Gal(M/L)$,
    the usual action of $\Gal(L/K)$ on the left-hand side corresponds
    the conjugacy action on $\Gal(M/L)$.
    Hence, the action on $J_L\big/L^\times N_{M/L}J_M$ is trivial
    if and only if $M$ is central over $L/K$.
    \item[(\ref{it:ideleloc})] Let $\p$ be the place of $K$ below $\P$,
    and $\widetilde{\P}$ be a place of $M$ above $\P$.

    Suppose that the decomposition group of $M/K$ at $\widetilde{\P}$,
    which is isomorphic to $\Gal(M_{\widetilde{\P}}/K_\p)$, is abelian.
    Let $\mathbf{a}\in H_{L/K,\P}$.
    The $\P$-component $\mathbf{a}_\P$ is in the kernel of
    $N_{L_\P/K_\p}\colon L_\P^\times\to K_\p^\times$,
    and the other components are $1$.
    Thus
    \[
      (\mathbf{a}_\P,M_{\widetilde{\P}}/L_\P)=(N_{L_\P/K_\p}\mathbf{a}_\P,M_{\widetilde{\P}}/K_\p)=1,
    \]
    it means $\art{M/L}{\mathbf{a}}=1$.
    Hence $\mathbf{a}\in L^\times N_{M/L}J_M$.

    Conversely, suppose $H_{L/K,\P}\subseteq L^\times N_{M/L}J_M$.
    Let $M_{\widetilde{\P}}'$ be the maximal abelian subextension of $M_{\widetilde{\P}}/K_\p$,
    and take arbitrary $\sigma\in\Gal(M_{\widetilde{\P}}/M_{\widetilde{\P}}')$.
    Then there is $\alpha\in L_\P^\times$ such that $\sigma=(\alpha,M_{\widetilde{\P}}/L_\P)$.
    We have
    \[
      (N_{L_\P/K_\p}\alpha,M_{\widetilde{\P}}'/K_\p)=\sigma|_{M_{\widetilde{\P}}'}=1.
    \]
    Therefore $N_{L_\P/K_\p}\alpha\in K_\p^\times$ is a norm from $M_{\widetilde{\P}}'^\times$,
    hence from $M_{\widetilde{\P}}^\times$.
    We write $N_{L_\P/K_\p}\alpha=N_{M_{\widetilde{\P}}/K_\p}\beta$ with $\beta\in M_{\widetilde{\P}}^\times$.
    Define an id\`ele $\mathbf{a}\in J_L$ as $\mathbf{a}_\P=\alpha N_{M_{\widetilde{\P}}/L_\P}\beta^{-1}$
    and $\mathbf{a}_{\P'}=1$ for $\P'\ne\P$.
    Then we have $\mathbf{a}\in H_{L/K,\P}$, hence $\mathbf{a}\in L^\times N_{M/L}J_M$.
    Therefore $\art{M/L}{\mathbf{a}}=1$, i.e.,
    \[
      \sigma=(\alpha N_{M_{\widetilde{\P}}/L_\P}\beta^{-1},M_{\widetilde{\P}}/L_\P)=1.
    \]
    Hence $M_{\widetilde{\P}}=M_{\widetilde{\P}}'$.
    \item[(\ref{it:idelef})] One can show that
    $H_{L/K}$ is topologically generated by $A_{L/K}$ and
    $H_{L/K,\P}$ for all $\P$.
    Hence from (\ref{it:idelecent}) and (\ref{it:ideleloc}) we have the result.
  \end{enumerate}
\end{proof}
\begin{proof}[Proof of Theorem~\ref{thm:cent}]
  Write $\Gal(L/K)=\{\sigma_1,\sigma_2,\dots,\sigma_n\}$ then
  the map
  \[
    J_L^n\longrightarrow A_{L/K};\;(\mathbf{a}_1,\mathbf{a}_2,\dots,\mathbf{a}_n)\longmapsto
    \frac{\sigma_1\mathbf{a}_1}{\mathbf{a}_1}\frac{\sigma_2\mathbf{a}_2}{\mathbf{a}_2}\dotsm\frac{\sigma_n\mathbf{a}_n}{\mathbf{a}_n}
  \]
  is continuous and surjective.
  Let $C_L=J_L/L^\times$ be the id\`ele class group of $L$.
  Let $D_L\subseteq C_L$ be the connected component of the identity element of $C_L$ and
  $\widetilde{D}_L\subseteq J_L$ be the inverse image of $D_L$.
  The above map induces a continuous surjection
  \[
    (C_L/D_L)^n\longrightarrow A_{L/K}\widetilde{D}_L/\widetilde{D}_L.
  \]
  Since $C_L/D_L$ is compact, $A_{L/K}\widetilde{D}_L/\widetilde{D}_L$ is also compact,
  hence closed in $J_L/\widetilde{D}_L\simeq C_L/D_L$.
  Therefore $A_{L/K}\widetilde{D}_L$ is closed in $J_L$.

  It is well-known that the $(-1)$-st Tate cohomology group
  \[
    \hat{H}^{-1}(\Gal(L/K),C_L)\simeq N^{-1}K^\times\Big/L^\times A_{L/K}
  \]
  is finite.
  Hence $N^{-1}K^\times\widetilde{D}_L$ and $H_{L/K}\widetilde{D}_L$,
  which have the subgroup $A_{L/K}\widetilde{D}_L$ of finite indices,
  are closed.
  Therefore from the above lemma,
  $N^{-1}K^\times\widetilde{D}_L$ corresponds to $K^\ab$ and
  $H_{L/K}\widetilde{D}_L$ corresponds to $F$ via the Artin map on $J_L$.

  We have proved that the lower map in the following diagram is an isomorphism:
  \[
    \begin{tikzpicture}[auto,scale=2]
      \node(NiK)at(0,1){$N^{-1}K^\times\Big/L^\times H_{L/K}$};
      \node(NiKD)at(0,0){$N^{-1}K^\times\widetilde{D}_L\Big/H_{L/K}\widetilde{D}_L$};
      \node(KNJ)at(2.5,1){$K^\times\cap NJ_L\Big/NL^\times$};
      \node(Gal)at(2.5,0){$\Gal(F/K^\ab)$.};
      \draw[->](NiK)--node{$N$}(KNJ);
      \draw[->](NiKD)--node{$\art{F/L}{\cdot}$}(Gal);
      \draw[->](NiK)--(NiKD);
      \draw[->](KNJ)--node{$\pi_{F/K}$}(Gal);
    \end{tikzpicture}
  \]
  The upper map is obviously an isomorphism.

  We can obtain $N^{-1}K^\times\cap\widetilde{D}_L\subseteq L^\times A_{L/K}$
  from $\hat{H}^{-1}(\Gal(L/K),D_L)=1$,
  hence we have
  \[
    N^{-1}K^\times\cap H_{L/K}\widetilde{D}_L=(N^{-1}K^\times\cap\widetilde{D}_L)H_{L/K}\subseteq L^\times H_{L/K}.
  \]
  This shows that the left vertical map in the above diagram is injective,
  hence an isomorphism.

  Take $\mathbf{a}\in N^{-1}K^\times$ and let $a=N\mathbf{a}\in K^\times$.
  We may assume that for each place $\p$ of $K$,
  there is a single place $\P_\p$ of $L$ above $\p$ such that
  $N_{L_{\P_\p}/K_\p}\mathbf{a}_{\P_\p}=a$ and $\mathbf{a}_\P=1$ for all $\P\ne\P_\p$ above $\p$.
  Then if $\widetilde{\P}$ is a place of $F$ above $\P_\p$,
  \[
    (a,F_{\widetilde{\P}}/K_\p)=(\mathbf{a}_{\P_\p},F_{\widetilde{\P}}/L_{\P_\p})
  \]
  is in the center of $\Gal(F/K)$, hence $F\subseteq K_a$.
  Also we have
  \[
    \pi_{F/L}(a)=\prod_\p(\mathbf{a}_{\P_\p},F_{\widetilde{\P}}/L_{\P_\p})=\art{F/L}{\mathbf{a}},
  \]
  which shows that the right vertical map in the above diagram is well-defined
  and the diagram is commutative.

  On the lower map of the diagram, $ZK^\ab$ corresponds to
  \[
    L^\times A_{L/K}U_L\cap N^{-1}K^\times\widetilde{D}_L.
  \]
  Since $\widetilde{D}_L\subseteq L^\times U_L$, it equals
  \[
    (U_L\cap N^{-1}K^\times)A_{L/K}\widetilde{D}_L=N^{-1}E_K\widetilde{D}_L.
  \]
  Therefore we obtain the following commutative diagram of isomorphisms:
  \[
    \begin{tikzpicture}[auto,scale=2]
      \node(NiK)at(0,1){$L^\times N^{-1}E_K\Big/L^\times H_{L/K}$};
      \node(NiKD)at(0,0){$N^{-1}E_K\widetilde{D}_L\Big/H_{L/K}\widetilde{D}_L$};
      \node(KNJ)at(2.5,1){$(E_K\cap NU_L)NL^\times\Big/NL^\times$};
      \node(Gal)at(2.5,0){$\Gal(F/ZK^\ab)$.};
      \draw[->](NiK)--node{$N$}(KNJ);
      \draw[->](NiKD)--node{$\art{F/L}{\cdot}$}(Gal);
      \draw[->](NiK)--(NiKD);
      \draw[->](KNJ)--node{$\pi_{F/K}$}(Gal);
    \end{tikzpicture}
  \]
\end{proof}

\begin{corollary}
  Let $l$ be a prime.
  In the same situation as Theorem~\ref{thm:cent},
  assume that $L/K$ is abelian $l$-extension.
  Let $F^{(l)}$ be the maximal elementary abelian $l$-subextension of $F/K^\ab$.

  Then the class number of $L$ is prime to $l$ if and only if,
  the genus number of $L$ is prime to $l$ and
  the image of $E_K\cap NU_L$ by $\pi_{F^{(l)}/K}$ equals $\Gal(F^{(l)}/K^\ab)$.
\end{corollary}
\begin{proof}
  Since $\pi_{F^{(l)}/K}=\pi_{F/K}|_{F^{(l)}}$,
  Theorem~\ref{thm:cent} yields that
  the image of $E_K\cap NU_L$ by $\pi_{F^{(l)}/K}$ equals $\Gal(F^{(l)}/F^{(l)}\cap ZK^\ab)$.

  Let $G$ be the genus field of $L/K$.
  If the class number of $L$ is prime to $l$,
  so are $[G:L]$ and $[ZK^\ab:K^\ab]$, hence $F^{(l)}\cap ZK^\ab=K^\ab$.
  Conversely, assume $[G:L]$ is prime to $l$ and $F^{(l)}\cap ZK^\ab=K^\ab$.
  Let $M$ be the maximal unramified elementary abelian $l$-extension of $L$,
  and write $\Sigma=\Gal(M/K)$ and $\Delta=\Gal(M/L)$.
  Then the corresponding field $M\cap Z$ of $[\Delta,\Sigma]$ is
  contained in $F^{(l)}\cap Z$, hence in $K^\ab\cap Z=G$ by the assumptions.
  Thus $M\cap Z=L$ and we obtain $\Delta=[\Delta,\Sigma]=[[\Delta,\Sigma],\Sigma]=\dotsb$,
  which yields $\Delta=1$ from the nilpotency of $\Sigma$.
  Hence the class number of $L$ is prime to $l$.
\end{proof}

We give some examples of the above corollary.
The next example is one of the known results on the parity problems.
\begin{example}
  Let $K=\Q$ and
  let $p,q,r\not\equiv3\ (4)$ be distinct primes with $(p/q)=(q/r)=(r/p)=-1$.

  Then the class number of $L=\Q(\sqrt p,\sqrt q,\sqrt r)$ is even if and only if
  \[
    \art{p}{q}_4\art{q}{r}_4\art{r}{p}_4=-1.
  \]
\end{example}
\begin{proof}
  From the genus formula \cite{nmj/1118802021},
  the genus number of $L/\Q$ equals $1$.

  Let $\sigma_p\ (p\text{ : prime}),\sigma_{-1},\sigma_\infty$ be those in the previous section.
  For a prime $l\ne p,q,r$ we have
  $\sigma_l|_{F^{(2)}}\in\Gal(F^{(2)}/L)$,
  since $l$ does not ramify in $L/\Q$.
  Similarly $\sigma_{-1}|_{F^{(2)}},\sigma_\infty|_{F^{(2)}}$ are also in the center of $\Gal(F^{(2)}/\Q)$,
  so we may ignore anything but $\sigma_p,\sigma_q,\sigma_r$.

  From Theorem~\ref{thm:graph} we have
  \[
    \{\{p,q\},\{p,r\}\},\{\{q,p\},\{q,r\}\},\{\{r,p\},\{r,q\}\}\in\mathcal{A}_{F^{(2)}/\Q}(-1).
  \]
  On the other hand, from Proposition~\ref{prop:d8} and \ref{prop:trip} with $a=-1,\,\mu=pq,\,\nu=pr$ we have
  \[
    \{\{p,q\},\{q,r\},\{r,p\}\}\in\mathcal{I}_{F^{(2)}/\Q}(-1).
  \]
  Hence $\mathcal{I}_{F^{(2)}/\Q}(-1)$ is a group of order two generated by $\{\{p,q\},\{q,r\},\{r,p\}\}$.
  It means that $\pi_{F^{(2)}/\Q}(-1)=1$ if and only if $c_{p,q}+c_{q,r}+c_{r,p}=0$
  where $c_{p,q}$ are those in the equality~\eqref{eq:qcpq}.
  Therefore we have the result from Theorem~\ref{thm:qmain}.
\end{proof}

The next example is an analogue of the previous example
with the base field as a quadratic field.
\begin{example}
  Let $K=\Q(\sqrt2)$.
  Put $\pi_2=2+\sqrt2,\,\pi_{41}=7+2\sqrt2$ and
  let $\pi\ne\pi_2,\pi_{41}$ be a prime element of $K$ such that $K(\sqrt\pi)/K$ is unramified outside $(\pi)$
  (Note that $\pi_2$ and $\pi_{41}$ also satisfy this property).
  Also assume
  \[
    \art{\pi_2,\pi}{(\pi)}=\art{\pi_{41},\pi}{(\pi)}=-1.
  \]
  Put
  \begin{align}
    \xi&=1+2\sqrt2+\sqrt{2\pi_2\pi_{41}}\in K(\sqrt{\pi_2\pi_{41}}),\\
    \eta&=3+2\sqrt2+\sqrt{\pi_2\pi_{41}}\in K(\sqrt{\pi_2\pi_{41}}).
  \end{align}

  Then the class number of $L=K(\sqrt{\pi_2},\sqrt{\pi_{41}},\sqrt\pi)$ is even if and only if
  \[
    \art{\xi,\pi}{\P}=-1\text{ and }\art{\eta,\pi}{\P}=1
  \]
  where $\P$ is a place of $K(\sqrt{\pi_2\pi_{41}})$ above $(\pi)$.
\end{example}
\begin{proof}
  The genus number of $L/K$ equals $1$.

  Let $\varepsilon=-1$ or $1+\sqrt2$.
  Define $\sigma_\p\ (\p\text{ : place of }K),\sigma_{-1},\sigma_{1+\sqrt2}\in\Gal(K_\varepsilon^{(2)}/K)$
  similarly as the case of $K=\Q$, i.e.,
  \begin{itemize}
    \item for $\p\ne(\pi_2)$, $\sigma_\p$ is a lift of $\art{g_\p,K^\ab/K}{\p}$
    where $g_\p$ is a local unit at $\p$ which is not square;
    \item $\sigma_{(\pi_2)},\sigma_{-1},\sigma_{1+\sqrt2}$ are lifts of
    \[
      \art{5,K^\ab/K}{(\pi_2)},\art{-1,K^\ab/K}{(\pi_2)},\art{1+\sqrt2,K^\ab/K}{(\pi_2)}
    \]
    respectively.
  \end{itemize}
  Note that the local unit group at $(\pi_2)$ is (topologically) generated by
  $-1,1+\sqrt2$ and $5$.
  From the property of $\pi$ we have
  \[
    \art{-1,L/K}{\p}=\art{1+\sqrt2,L/K}{\p}=1
  \]
  for any place $\p$ of $K$.
  Hence, similarly as the previous example,
  we may ignore anything but $\sigma_{(\pi_2)},\sigma_{(\pi_{41})},\sigma_{(\pi)}$.

  We can determine $\mathcal{I}_{F^{(2)}/K}(\varepsilon)$
  without determining $\mathcal{I}_{K_\varepsilon^{(2)}/K}(\varepsilon)$.
  A similar argument as in the proof of Theorem~\ref{thm:graph} yields
  \begin{align}
    \{\{(\pi_2),(\pi_{41})\},\{(\pi_2),(\pi)\}\}&,\\
    \{\{(\pi_{41}),(\pi_2)\},\{(\pi_{41}),(\pi)\}\}&,\\
    \{\{(\pi),(\pi_2)\},\{(\pi),(\pi_{41})\}\}&\in\mathcal{A}_{F^{(2)}/K}(\varepsilon).
  \end{align}
  On the other hand, from Proposition~\ref{prop:d8} and \ref{prop:trip}
  with $a=\varepsilon,\,\mu=\pi_2\pi_{41},\,\nu=\pi_2\pi$ we have
  \[
    \{\{(\pi_2),(\pi_{41})\},\{(\pi_{41}),(\pi)\},\{(\pi),(\pi_2)\}\}\in\mathcal{I}_{F^{(2)}/K}(\varepsilon).
  \]
  Hence we have
  \[
    \mathcal{I}_{F^{(2)}/K}(-1)=\mathcal{I}_{F^{(2)}/K}(1+\sqrt2),
  \]
  which is a group of order two generated by $\{\{(\pi_2),(\pi_{41})\},\{(\pi_{41}),(\pi)\},\{(\pi),(\pi_2)\}\}$.

  Put
  \begin{itemize}
    \item $b=3(1+\sqrt2),\,\alpha=\xi/b$ for $\varepsilon=-1$;
    \item $b=1-\sqrt2,\,\alpha=\eta/b$ for $\varepsilon=1+\sqrt2$;
  \end{itemize}
  then one can check that these satisfy the assumptions of Proposition~\ref{prop:main}.
  Hence from Proposition~\ref{prop:main}, $c_{(\pi_2),(\pi_{41})}+c_{(\pi_{41}),(\pi)}+c_{(\pi),(\pi_2)}=0$
  if and only if
  \[
    \art{b\alpha,\pi_2\pi}{\P}=\art{b,\pi_2\pi}{(\pi_2)},
  \]
  from the assumptions on $b$ and $\alpha$ it means
  \[
    \art{b\alpha,\pi}{\P}=\art{b,\pi_2}{(\pi_2)}.
  \]
  Since $\art{3(1+\sqrt2),\pi_2}{(\pi_2)}=-1$ and $\art{1-\sqrt2,\pi_2}{(\pi_2)}=1$
  we have the result.
\end{proof}

\section{Application on fundamental units of real quadratic fields}

Using Proposition~\ref{prop:main} several times, we can obtain some relations modulo square
on fundamental units of real quadratic fields.

\begin{theorem}
  \label{thm:sq}
  Let $m_1,\dots,m_k$ be square-free positive integers and
  let $p_1,\dots,p_t$ be all the prime divisors of $m_1\dotsm m_k$.
  Assume $N\varepsilon_{m_i}=-1$ for each $i$,
  and also assume that $m_1\dotsm m_k$ is a perfect square.

  Then $\varepsilon_{m_1}\dotsm\varepsilon_{m_k}$ is square in $\Q(\sqrt{p_1},\dots,\sqrt{p_t})$.
\end{theorem}
\begin{proof}
  Let $E=\Q(\sqrt{p_1},\dots,\sqrt{p_t})$.
  Let $l$ be a prime which splits totally in $E(\sqrt{-1})$,
  equivalently, satisfies $l\equiv1\ (4)$ and $(p_j/l)=1$ for all $j$.
  We can apply Proposition~\ref{prop:d8}, \ref{prop:trip}, \ref{prop:main}
  with $K=\Q,\,\mu=m_i,\,\nu=l,\,\alpha=\varepsilon_{m_i},\,b=1$,
  since $p_j\not\equiv3\ (4)$ for all $j$ from our assumption.
  We have
  \[
    \sum_{p|m_i}c_{p,l}=0\iff\art{\varepsilon_{m_i}}{l}=1.
  \]
  Also we have
  \[
    \sum_{i=1}^k\sum_{p|m_i}c_{p,l}=0
  \]
  since $m_1\dotsm m_k$ is square.
  Therefore we obtain
  \[
    \art{\varepsilon_{m_1}\dotsm\varepsilon_{m_k}}{l}=1
  \]
  hence $l$ splits totally in $E(\sqrt{-1},\sqrt{\varepsilon_{m_1}\dotsm\varepsilon_{m_k}})$,
  which is Galois over $\Q$.
  From Chebotarev density theorem we have
  \[
    E(\sqrt{-1},\sqrt{\varepsilon_{m_1}\dotsm\varepsilon_{m_k}})=E(\sqrt{-1}),
  \]
  which yields $\sqrt{\varepsilon_{m_1}\dotsm\varepsilon_{m_k}}\in E$
  since $\varepsilon_{m_1}\dotsm\varepsilon_{m_k}>0$.
\end{proof}

We also have an analogous theorem of Theorem~\ref{thm:sq} for fundamental units with positive norms:
\begin{theorem}
  Let $m$ be a square-free positive integer and
  let $p_1,\dots,p_t$ be all the prime divisors of $m$.
  Assume $N\varepsilon_m=1$.

  Then $\varepsilon_m$ is square in $E$, where
  \[
    E=\begin{cases}
      \Q(\sqrt{p_1},\dots,\sqrt{p_t})&(m\not\equiv3\pmod4)\\
      \Q(\sqrt2,\sqrt{p_1},\dots,\sqrt{p_t})&(m\equiv3\pmod4).
    \end{cases}
  \]
\end{theorem}
\begin{proof}
  Let $l$ be a prime which splits totally in $E(\sqrt{-1})$.
  We can apply Proposition~\ref{prop:d8}, \ref{prop:trip}, \ref{prop:main}
  with $K=\Q,\,\mu=m,\,\nu=l,\,\alpha=\varepsilon_m,\,b=1$.
  Note that all the invariants are $0$, since $a=1$ so that all Artin symbols are identity.
  Thus we have
  \[
    \art{\varepsilon_m}{l}=1.
  \]
  Similarly as Theorem~\ref{thm:sq},
  $\sqrt{\varepsilon_m}\in E$ holds.
\end{proof}

We would like to strengthen the above theorems.
In the situation of Theorem~\ref{thm:sq},
let $F=\Q(\sqrt{m_1},\dots,\sqrt{m_k})$
then $F(\sqrt{\varepsilon_{m_1}\dotsm\varepsilon_{m_k}})$ is
a quadratic or trivial subextension of $F$ in $E$.
We want to identify which of such fields
$F(\sqrt{\varepsilon_{m_1}\dotsm\varepsilon_{m_k}})$ is.

Proposition~\ref{prop:main} can also be used to narrow down
the candidates for $F(\sqrt{\varepsilon_{m_1}\dotsm\varepsilon_{m_k}})$.
The following theorem gives a general framework
for narrowing down the candidates.

\begin{theorem}
  \label{thm:candm}
  Let $m_1,\dots,m_k,n_1,\dots,n_k$ be square-free positive integers satisfying
  \begin{enumerate}
    \item $m_1\dotsm m_kn_1\dotsm n_k$ is square,
    \item $\gcd(m_i,n_i)=1$ and $N\varepsilon_{m_in_i}=-1$ for all $i$.
  \end{enumerate}
  Suppose that there exists a set $P$ of primes such that for all $i$,
  \begin{enumerate}
    \item $p\in P\iff(n_i/p)=-1$ for $p\mid m_i$,
    \item $q\in P\iff(m_i/q)=-1$ for $q\mid n_i$.
  \end{enumerate}
  Let $p_1,\dots,p_t$ be all the prime divisors of $m_1\dotsm m_kn_1\dotsm n_k$.
  Put
  \[
    F=\Q(\sqrt{m_1n_1},\dots,\sqrt{m_kn_k})
  \]
  and $F(\sqrt{\varepsilon_{m_1n_1}\dotsm\varepsilon_{m_kn_k}})=F(\sqrt d)$
  where $d\mid p_1\dotsm p_t$.
  Let $D$ be the set of all prime divisors of $d$.

  Then the cardinality of $D\cap P$ is even if and only if
  \[
    \sum_{i=1}^k\sum_{p|m_i}\sum_{q|n_i}c_{p,q}=0
  \]
  where $c_{p,q}$ are those in the equality~\eqref{eq:qcpq}.
\end{theorem}
\begin{proof}
  From Dirichlet's theorem,
  there exists a prime $l\equiv1\ (4)$ such that
  for $j=1,\dots,t$,
  $(p_j/l)=-1$ if and only if $p_j\in P$.
  This yields $(p/l)=(n_i/p)$ for $p\mid m_i$ and
  $(q/l)=(m_i/q)$ for $q\mid n_i$ from the assumption on $P$.
  From this it is easy to show that $(m_in_i/l)=1$ for all $i$,
  hence $l$ totally splits in $F$.
  Also it is easy to show that $\art{m_in_i,m_il}{p}=1$
  for any places $p$,
  thus we can use Proposition~\ref{prop:d8}, \ref{prop:trip}, \ref{prop:main}
  with $K=\Q,\,\alpha=\varepsilon_{m_in_i},\,b=1,\,\mu=m_in_i,\,\nu=m_il$,
  which yields
  \[
    \art{\varepsilon_{m_in_i}}{l}=1\iff\sum_{p|m_in_i}c_{p,l}+\sum_{p|m_i}\sum_{q|n_i}c_{p,q}=0.
  \]
  Since $\prod_im_in_i$ is square, we obtain
  \[
    \art{d}{l}=1\iff\art{\prod_i\varepsilon_{m_in_i}}{l}=1\iff\sum_{i=1}^k\sum_{p|m_i}\sum_{q|n_i}c_{p,q}=0.
  \]
\end{proof}

\begin{example}
  Let $p,q,r\not\equiv3\ (4)$ be distinct primes
  with $(p/q)=(q/r)=(r/p)=-1$.
  Applying Theorem~\ref{thm:candm} with
  $k=3,\,m_1=n_3=p,\,m_2=n_1=q,\,m_3=n_2=r$ and $P=\{p,q,r\}$,
  we have that
  $\varepsilon_{pq}\varepsilon_{qr}\varepsilon_{rp}$ is square in
  $\Q(\sqrt{pq},\sqrt{pr})$
  if and only if
  $(p/q)_4(q/r)_4(r/p)_4=-1$.
\end{example}

We also have an analogous theorem of Theorem~\ref{thm:candm}:
\begin{theorem}
  \label{thm:candp}
  Let $m,n$ be square-free positive integers satisfying
  \begin{enumerate}
    \item $mn$ is square,
    \item $\gcd(m,n)=1$ and $N\varepsilon_{mn}=1$,
    \item $p\equiv1\ (4)$ for each prime $p\mid m$.
  \end{enumerate}
  Define a set $P$ of primes to be
  \begin{align}
    \set{p\text{ : prime}|p\mid m,\,(n/p)=-1}&\cup
    \set{q\text{ : prime}|q\mid n,\,(m/q)=-1}\\
    &\cup\begin{cases}
      \{2\}&(m\equiv5\ (8)\text{ and }n\equiv3\ (4))\\
      \emptyset&(\text{otherwise}).
    \end{cases}
  \end{align}
  Let $F=\Q(\sqrt{mn})$
  and $F(\sqrt{\varepsilon_{mn}})=F(\sqrt d)$,
  where $d$ is a square-free positive integer such that
  \[
    \begin{cases}
      d\mid2mn&(n\equiv3\ (4))\\
      d\mid mn&(\text{otherwise}).
    \end{cases}
  \]
  Let $D$ be the set of all prime divisors of $d$.

  Then the cardinality of $D\cap P$ is even.
\end{theorem}
\begin{proof}
  Take a prime $l\equiv1\ (4)$ such that
  for $p\mid2mn$,
  $(p/l)=-1$ if and only if $p\in P$.
  A similar argument as in Theorem~\ref{thm:candm} yields
  $(\varepsilon_{mn}/l)=1$,
  hence $(d/l)=1$.
\end{proof}

The sign of the norm of a fundamental unit may be determined
from Theorem~\ref{thm:candm} or \ref{thm:candp}:
\begin{corollary}
  Let $m,n$ be square-free positive integers whose prime divisors are not equivalent to $3\bmod4$.
  Assume that
  \begin{enumerate}
    \item $\gcd(m,n)=1$,
    \item $(n/p)=1$ for all $p\mid m$,
    \item $(m/q)=1$ for all $q\mid n$,
    \item $\prod\limits_{p|m}(n/p)_4\prod\limits_{q|n}(m/q)_4=-1$.
  \end{enumerate}
  Then $N\varepsilon_{mn}=1$.
\end{corollary}
\begin{proof}
  Factorize $m=p_1\dotsm p_s$ and $n=q_1\dotsm q_t$.
  Suppose $N\varepsilon_{mn}=-1$ and apply Theorem~\ref{thm:candm}
  with $k=1+s+t,\,m_1=m,\,n_1=n,\,m_{1+j}=p_j\ (j=1,\dots,s),\,m_{1+s+j}=q_j\ (j=1,\dots,t),\,n_{1+j}=1\ (j=1,\dots,s+t)$
  and $P=\emptyset$,
  we have $\sum_{p|m}\sum_{q|n}c_{p,q}=0$.
  Applying Theorem~\ref{thm:qmain} for the complete bipartite graph on
  $p_1,\dots,p_s$ and $q_1,\dots,q_t$ we have
  $\prod_{p|m}(n/p)_4\prod_{q|n}(m/q)_4=1$,
  which contradicts our assumption.
\end{proof}

\begin{example}
  Let $p,q,r,s\not\equiv3\ (4)$ be distinct primes
  with $(p/q)=(q/r)=(r/s)=(s/p)$ and
  $(pr/q)_4(pr/s)_4(qs/p)_4(qs/r)_4=-1$.
  Then $N\varepsilon_{pqrs}=1$.
\end{example}

Finally we describe an application to the unit index of
Kuroda's class number formula.
The following proposition is a special case of
Kuroda's class number formula,
and we want to determine $Q(F)$ in it.
Here $h_K$ denotes the class number of a number field $K$.
\begin{proposition}[\cite{kuroda_1950,nmj/1118799770}]
  Let $F$ be a real biquadratic field and $K_1,K_2,K_3$ be
  the quadratic subfields of $F$.
  Then
  \[
    h_F=\frac{1}{4}Q(F)h_{K_1}h_{K_2}h_{K_3}
  \]
  where $Q(F)=[E_F:E_{K_1}E_{K_2}E_{K_3}]$.
\end{proposition}

For example, we can deduce the following formula from our results:
\begin{example}
  Let $p,q,r\not\equiv3\ (4)$ be distinct primes with $(p/q)=1$ and $(q/r)=-1$,
  and put $F=\Q(\sqrt p,\sqrt{qr})$.
  Then
  \[
    Q(F)=\begin{cases}
      1&(N\varepsilon_{pqr}=-1\text{ and }(p/q)_4(q/p)_4=-1)\\
      2&(\text{otherwise}).
    \end{cases}
  \]
\end{example}
\begin{proof}
  Let $K_1=\Q(\sqrt p),\,K_2=\Q(\sqrt{qr}),\,K_3=\Q(\sqrt{pqr})$.
  The quotient $E_F/E_{K_1}E_{K_2}E_{K_3}$ is an elementary abelian $2$-group
  (see \cite{nmj/1118799770}).
  Hence one can show that the squaring map induces an embedding
  \[
    E_F\Big/E_{K_1}E_{K_2}E_{K_3}\longrightarrow E_{K_1}E_{K_2}E_{K_3}\Big/(E_{K_1}E_{K_2}E_{K_3})^2.
  \]
  To obtain the size of the left-hand side,
  we check that for each element of the right-hand side whether it is square in $F$.

  Suppose $N\varepsilon_{pqr}=-1$.
  It suffices to consider the totally positive units,
  hence we have to check whether $\varepsilon_p\varepsilon_{qr}\varepsilon_{pqr}$ is square in $F$.
  Using Theorem~\ref{thm:candm} with
  $k=3,\,m_1=p,\,n_1=1,\,m_2=q,\,n_2=r,\,m_3=pr,\,n_3=q$ and $P=\{q,r\}$,
  we obtain
  \[
    F(\sqrt{\varepsilon_p\varepsilon_{qr}\varepsilon_{pqr}})=\begin{cases}
      F&(c_{p,q}=0)\\F(\sqrt q)&(c_{p,q}=1).
    \end{cases}
  \]
  Therefore $Q(F)=1$ if and only if $c_{p,q}=1$, i.e.,
  $(p/q)_4(q/p)_4=-1$ by Theorem~\ref{thm:qmain}.

  Suppose $N\varepsilon_{pqr}=1$.
  Similarly, we have to check whether $\varepsilon_{pqr}$ is square in $F$.
  Using Theorem~\ref{thm:candp} with $m=pr,\,n=q$,
  we obtain
  \[
    F(\sqrt{\varepsilon_{pqr}})=F.
  \]
  Therefore $Q(F)=2$.
\end{proof}

\nocite{MR819231,MR2392026}

\bibliographystyle{plain}
\bibliography{ref}

\end{document}